\documentclass[11pt]{amsart}
\usepackage{CJKutf8}
\usepackage{amsmath,amsthm,amssymb} 
\usepackage{url}
\usepackage{bm}
\usepackage{esint}
\usepackage{afterpage}
\usepackage{centernot}
\usepackage[T1]{fontenc} 
\usepackage{textcomp} 
\usepackage{footnote}
\usepackage{color}
\usepackage{mathrsfs}
\usepackage{leftidx}
\usepackage{bm}
\usepackage{url}
\usepackage{comment}
\usepackage{gensymb}
\usepackage{tikz}
\usepackage{pgfplots}
\usetikzlibrary{calc}
\usepackage{graphics}
\usepackage{tikz-cd}
\usetikzlibrary{arrows} 
\usepackage{epsf}
\usepackage[bookmarksnumbered,pdfpagelabels=true,plainpages=false,colorlinks=true,linkcolor=black,citecolor=red,urlcolor=blue]{hyperref}

\DeclareMathOperator{\Diff}{Diff}

\DeclareMathOperator{\SU}{SU}

\DeclareMathOperator{\SL}{SL}

\DeclareMathOperator{\Sym}{Sym}

\newcommand{\C}{\mathbb C}

\newcommand{\N}{\mathbb N}

\newcommand{\id}{\mathrm id}

\def\C{\mathbb{C}}
\def\N{\mathbb{N}}

\def\cL{{\mathcal L}}

\newcommand{\diff}{\text{\rm d}}

\theoremstyle{plain}

\newtheorem{theorem}{Theorem}
\newtheorem{proposition}[theorem]{Proposition}
\newtheorem{lemma}[theorem]{Lemma}

\newtheorem{conjecture}[theorem]{Conjecture}

\theoremstyle{definition}
\newtheorem{definition}[theorem]{Definition}
\newtheorem{remark}[theorem]{Remark}

\newtheorem{example}[theorem]{Example}

\theoremstyle{plain}
\newtheorem*{solution*}{Solution}
\newtheorem*{theorem*}{Theorem}
\newtheorem*{proposition*}{Proposition}
\newtheorem*{lemma*}{Lemma}
\newtheorem*{corollary*}{Corollary}
\newtheorem*{conjecture*}{Conjecture}
\theoremstyle{definition}
\newtheorem*{definition*}{Definition}
\newtheorem*{remark*}{Remark}
\newtheorem*{remarks*}{Remarks}
\numberwithin{equation}{section}
\numberwithin{theorem}{section}

\author[J. Fine]{Joel Fine}
\address{D\'epartement de Math\'ematiques, Universit\'e libre de Bruxelles, Brussels, Belgium}
\email{joel.fine@ulb.be}
\author[W.-Y. He]{Weiyong He}
\address{Department of Mathematics, University of Oregon, Eugene, OR 97403, USA}
\email{whe@uoregon.edu}
\author[C.-J. Yao]{Chengjian Yao}
\address{Institute of Mathematical Sciences, ShanghaiTech University, Pudong New District,
	Shanghai, 201210, China.}
\email{yaochj@shanghaitech.edu.cn}

\begin{document}
	\title{Hypersymplectic Structures Invariant Under an Effective Circle Action}

\begin{abstract}
	A hypersymplectic structure on a 4-manifold is a triple of symplectic forms for which any non-zero linear combination is again symplectic. In 2006, Donaldson \cite{Don0} conjectured that on a compact 4-manifold any hypersymplectic structure can be deformed through cohomologous hypersymplectic structures to a hyperkähler triple. We prove this under the assumption that the initial structure is invariant under an effective $S^1$-action. In particular we show that the underlying 4-manifold is diffeomorphic to $\mathbb{T}^4$. 
\end{abstract}

\maketitle

\section{Introduction}

Let $X$ be a 4-manifold. A \emph{hypersymplectic structure} on $X$ is a triple $\underline{\omega} = (\omega_1, \omega_2, \omega_3)$ of symplectic forms with the property that every non-zero linear combination $a^i \omega_i$ is again symplectic. (We use the summation convention throughout this article, summing repeated indices over 1,2,3.) In \cite{Don0}, Donaldson made the following conjecture (cf.~\cite[Conjecture 1.1]{FY}).

\begin{conjecture}\label{Donaldson-conjecture}
Let $\underline{\omega} = (\omega_1, \omega_2, \omega_3)$ be a hypersymplectic structure on a compact 4-manifold $X$ with $\int \omega_i \wedge \omega_j = 2 \delta_{ij}$. Then $\underline{\omega}$ can be deformed through cohomologous hypersymplectic structures to the triple of K\"ahler forms coming from a hyperk\"ahler metric on $X$. In particular, $X$ is diffeomorphic to $\mathbb{T}^4$ or a K3 surface. 
\end{conjecture}

The above conjecture remains largely open. To place this conjecture in a broader context, note that a hypersymplectic manifold automatically has zero first Chern class (see, e.g., the second page of \cite{FY} and \cite[Lemma 4]{FY2}, essentially because that the orthogonal complement of $\omega_1$ in $\text{Span}\{\omega_1, \omega_2,\omega_3\}$ with respect to the wedge product gives the trivial canonical bundle of $\left(X,J_1\right)$ for almost complex structure $J_1$ compatible with $\omega_1$). The classification of symplectic 4-manifolds with $c_1=0$ is an important problem that, despite much progress, appears to be currently out of reach. Donaldson's conjecture can be seen as a special and hopefully more tractable case of this much more difficult problem. In \cite{FY}, the first and the third authors introduced a geometric flow, called the \emph{hypersymplectic flow}, to attack Conjecture~\ref{Donaldson-conjecture}. The strategy is to evolve the given hypersymplectic structure in a canonical way towards a hyperk\"ahler structure.  The main problem is then to show that the flow exists for all time and converges to a hyperk\"ahler structure. However, it is challenging to understand the long time existence and convergence of the flow in general; see \cite{FY2}. To date, the flow has only been used to confirm the conjecture in very restricted and symmetric cases, with the main feature of either assuming an explicit $\mathbb{T}^3$-symmetry  on $\mathbb{T}^4$ \cite{FHY, HWY}, or assuming the hypersymplectic structure is of ``K\"ahler type'' \cite{PS}, as explained in \cite[Section 1.3]{FHY}. 

The main result of this paper is to prove Conjecture~\ref{Donaldson-conjecture} in the special case when the hypersymplectic structure is invariant under an effective circle action on a compact $4$-manifold.
\begin{theorem}\label{main-theorem}
    Let $X$ be a compact smooth $4$-manifold admitting a smooth hypersymplectic structure with $\int \omega_i \wedge \omega_j = 2 \delta_{ij}$. Suppose moreover that $\underline{\omega}$ is invariant under an effective smooth $S^1$-action. Then $X$ is diffeomorphic to $\mathbb{T}^4$. Moreover, there is a linear isotopy from $\underline\omega$ to a hyperk\"ahler triple through a family of hypersymplectic structures in the same cohomology class.
\end{theorem}

Our proof does not rely on the sophisticated machinery of the hypersymplectic flow. 
One difficult point in the proof of Donaldson's conjecture is that $X$ must be shown to be diffeomorphic to either $\mathbb{T}^4$ or a $K3$ surface. Assume for the moment that the $S^1$-action preserving $\underline{\omega}$ is free. Our first observation is the invariant hypersymplectic structure on $X$ induces a flat Riemannian metric on $X/S^1$, which is subsequently identified as a flat torus $\mathbb{T}^3$. A topological argument then implies that $X \to \mathbb{T}^3$ is a trivial $S^1$-bundle and so $X$ is diffeomorphic to $\mathbb{T}^4$. 

Next, by exploiting the $S^1$-symmetry and using the idea of a multi-moment map and the generalized Gibbons-Hawking construction for hypersymplectic structure introduced by Donaldson in \cite{Don}, we deform the hypersymplectic structure on $\mathbb{T}^4$ to a hyperk\"ahler structure via \emph{linear interpolation}. Finally, the assumption of \emph{freeness} of the action is weakened to \emph{effectiveness} by combining the known homological information for symplectic Calabi-Yau surfaces \cite{Li} and the observation that any possible zero of the generating vector field (of the effective circle action) must be isolated.

Despite the fact we have confirmed Conjecture~\ref{Donaldson-conjecture} in this symmetric case via a direct approach, we still believe the hypersymplectic flow is the most suitable method to attack the conjecture in general. 

We take the chance to clarify that the terminology ``hypersymplectic structure'' appearing in \cite{Don, FHY, FY, FY2, HWY} and the current article has a different meaning with the perhaps more well-known notion of ``hypersymplectic structure'' introduced by Hitchin \cite{Hitchin} or its equivalent ``neutral hyperk\"ahler structure'' \cite{Ka}. They certainly share some common feature, while there are two main differences: firstly, the signature (of the related Riemannian metric) is $(4,0)$ in our setting while Hitchin's hypersymplectic structure is of signature $(2,2)$; secondly, the triple of symplectic structures in our setting is only required to satisfy the positivity 
\begin{align*}
\left(\frac{\omega_i\wedge\omega_j}{2\mu}\right)>0
\end{align*}
 while for Hitchin's hypersymplectic structure the triple of symplectic structures $(\omega_1, \omega_2, \omega_3)$ (see \cite{DS}) satisfies the identity
 \begin{align*}
 \left(\frac{\omega_i\wedge\omega_j}{2\mu}\right)
 = 
 \text{diag}\left(-1,1,1\right) 
 \end{align*}
for some volume form $\mu$. 

\subsection*{Funding}
This work was supported by the “Excellence of Science” [40007524 to J.F.]; the Fund for Scientific Research [PDR T.0082.21 to J.F.]; and the National Natural Science Foundation of China [12401071 to C.Y.].

\subsection*{Acknowledgments}
The third author would like to thank Andries Salm for useful discussions. We are grateful to the referee for bringing Hitchin's article on hypersymplectic structures to our attention.

\section{Hypersymplectic Structure Invariant Under a Free $S^1$-Action}

Lert $X$ be a compact smooth $4$-manifold with a hypersymplectic structure $\underline{\omega}=(\omega_1, \omega_2,\omega_3)$. In this section we assume that $\underline{\omega}$ is invariant under a \emph{free} $S^1$-action, generated by the nowhere-vanishing vector field $\bm v$ on $X$.

For any $\left(a^1,a^2,a^3\right)\in \mathbf{R}^3\backslash\{0\}$, the $2$-form $a^i
\omega_i$ is non-degenerate at any point of $X$ (by the definition of hypersymplectic structure). This implies that the 1-form $a^i\iota_{\bm v}\omega_i$ is nonzero at any point of $X$. The condition $\mathcal{L}_{\bm v}\omega_i = \diff \iota_{\bm v}\omega_i =0$ now implies that the three $1$-forms $\alpha^i:= \iota_{\bm v}\omega_i$ (with $i=1,2,3$) are closed on $X$. Following Donaldson \cite{Don}, we use this triple of closed 1-forms to describe the triple $\underline\omega$ in a canonical way.

\begin{lemma}    \label{representation:canonical}
    There exists a unique tuple $\left(\alpha,\sigma\right)=\left( \alpha,\left(\sigma_{ij}\right)_{3\times 3}\right)\in \Gamma\left(T^*X\right)\times \Gamma\left(\underline\Sym^2\mathbf{R}^3\right)$ satisfying $\alpha\left(\bm v\right)\equiv 1$ on $X$ such that 
    \begin{align}
        \omega_i
        =
        \alpha\wedge \alpha^i + \sum_{(pqr)=(123)} \sigma_{ip}\alpha^q\wedge \alpha^r, \; \; i=1,2,3,
    \end{align}  
    where $\underline\Sym^2\mathbf{R}^3$ denotes the trivial bundle of symmetric $3\times 3$ matrices on $X$.
\end{lemma}

\begin{proof}Firstly, notice that $\alpha^i(\bm v)=\omega_i(\bm v,\bm v)=0$ for $i=1,2,3$. Fix any $x\in X$, take any $\alpha_x\in T^*_x X$ such that $\left\{\alpha_x, \alpha^1_x, \alpha^2_x,\alpha^3_x\right\}$ is a basis of $T^*_xX$ and $\alpha_x(\bm v_x)=1$. The set 
\begin{align*}
\left\{\alpha_x\wedge\alpha^1_x, \alpha_x\wedge\alpha^2_x, \alpha_x\wedge\alpha^3_x, \alpha^2_x\wedge\alpha^3_x, \alpha^3_x\wedge\alpha^1_x, \alpha^1_x\wedge\alpha^2_x\right\}
\end{align*}
is then a basis of $\Lambda^2T^*_xX$. There exists a $3\times 3$ matrix $A=\left(A_{ij}\right)$ such that
\begin{align*}
\omega_i|_x = \alpha_x\wedge \alpha^i_x + \sum_{(pqr)=(123)} A_{ip}\alpha^q_x\wedge \alpha^r_x.
\end{align*}
Let $\sigma_{ip}=\frac{1}{2}\left( A_{ip}+A_{pi}\right)$ and $C_k=\frac{1}{2}\epsilon_{ijk}A_{ij}$, then $A_{ip}=\sigma_{ip} + \epsilon_{ipq}C_q$ and as a consequence,
\begin{align*}
\omega_i|_x = \left(\alpha_x + C_k \alpha^k_x\right)\wedge \alpha^i_x + \sum_{(pqr)=(123)} \sigma_{ip}\alpha^q_x\wedge \alpha^r_x
\end{align*}
with $\sigma_{ij}=\sigma_{ji}$ and $\left(\alpha_x + C_k \alpha^k_x\right)(\bm v_x)=1$. This means we can write $\omega_i$'s in the  form claimed in the lemma. 

Next, we show such decomposition is unique. Suppose $\widetilde\alpha_x\in T^*_x X$ and $\widetilde{\sigma}=\left(\widetilde\sigma_{ij}\right)$ satisfying $\widetilde\sigma_{ij}=\widetilde\sigma_{ji}$ and $\widetilde\alpha_x\left(\bm v_x\right)=1$ are another choice of data for which 
\begin{align*}
\omega_i|_x = \widetilde\alpha_x\wedge \alpha^i_x + \sum_{(pqr)=(123)} \widetilde\sigma_{ip}\alpha^q_x\wedge\alpha^r_x, \;\; i=1,2,3.
\end{align*}
Let $\widetilde\alpha_x=\alpha_x + \widetilde{C}_q\alpha^q_x$. Since $\left\{\alpha_x\wedge\alpha^1_x, \alpha_x\wedge\alpha^2_x, \alpha_x\wedge\alpha^3_x, \alpha^2_x\wedge\alpha^3_x, \alpha^3_x\wedge\alpha^1_x, \alpha^1_x\wedge\alpha^2_x\right\}$ is a basis of $\Lambda^2T^*_xX$, we have $\widetilde\sigma_{ij}+\epsilon_{ijk}\widetilde{C}_k=\sigma_{ij}$ which implies $\widetilde C_k=0$ and consequently $\widetilde\alpha=\alpha, \widetilde\sigma=\sigma$. 
\end{proof}

We will use the notation $\alpha^{ij} = \alpha^i \wedge \alpha^j$ and $\alpha^{123} = \alpha^1 \wedge \alpha^2 \wedge \alpha^3$. We recall from~\cite{FY} that a hypersymplectic structure $\underline{\omega}$ determines a Riemannian metric $g_{\underline{\omega}}$ on $X$, via the formula
\begin{equation}
        g_{\underline\omega}(\bm u, \bm w)\mu_{\underline\omega}
        = 
        \frac{1}{6}\epsilon^{ijk}\iota_{\bm u}\omega_i\wedge\iota_{\bm w}\omega_j\wedge\omega_k.
	\label{g-formula}
\end{equation}
We write $\mu_{\underline{\omega}}$ for the volume form of this metric. We also use the notation from~\cite{FY} that $Q_{ij} = \frac{1}{2} g_{\underline{\omega}}(\omega_i,\omega_j)$. The metric is chosen so that the $\omega_i$'s are all self-dual, which fixes the conformal structure. The fact that $\underline{\omega}$ is hypersymplectic implies that $Q$ is positive-definite, and the volume form is chosen so that $\det Q = 1$.

\begin{remark}
Even though the hypersymplectic structure $\underline\omega$ naturally orients $X^4$ and the $S^1$-action preserves this orientation, it is still possible that $\alpha, \alpha^1, \alpha^2,\alpha^3$ is a \emph{negatively} oriented coframe.  In this case, the hypersymplectic structure $\omega_2, \omega_1, \omega_3$, formed by swapping the first two symplectic forms, will be a new hypersymplectic structure (defining the same orientation on $X^4$ as the original one) that is invariant under the given $S^1$-action, for which $\alpha, \alpha^1,\alpha^2,\alpha^3$ is \emph{positively} oriented. We assume from now on that $\underline{\omega}$ has the property that $\alpha, \alpha^1, \alpha^2, \alpha^3$ is positively oriented.
\end{remark}

With this  in hand, we now relate $g_{\underline{\omega}}$, $\mu_{\underline{\omega}}$, and $Q$ to the data $(\alpha, \sigma)$ from Lemma~\ref{representation:canonical}. 

\begin{lemma}
\label{lemma:metric-tensor}
    Write $V=\left(\det\sigma\right)^\frac{1}{3}$. Then 
\begin{equation} 
\begin{split}
    &\mu_{\underline\omega} = V\alpha\wedge \alpha^{123},\;\; Q_{ij}=V^{-1}\sigma_{ij},\\
&g_{\underline\omega}
= 
V^{-1} \alpha^2 + V Q^{-1}_{ij}\alpha^i\otimes \alpha^j,
\end{split}
\end{equation}
where $Q_{ij}^{-1}$ is the $(ij)$-th entry of the matrix $Q^{-1}$. In particular, $V>0$ and $\left(\sigma_{ij}\right)>0$ on $X$.
\end{lemma}

\begin{proof}
These are direct computations. By Lemma~\ref{representation:canonical},
   \begin{align}
        \omega_i\wedge\omega_j 
        = 
    2\sigma_{ij}\alpha\wedge\alpha^{123},
    \end{align}
    which, since $\underline{\omega}$ is hypersymplectic and $\alpha, \alpha^1,\alpha^2,\alpha^3$ is positively oriented, shows that $(\sigma_{ij}) >0$. We now have that $Q_{ij}=V^{-1}\sigma_{ij}$ with $V=\left(\det\sigma\right)^\frac{1}{3}$ (so that $\det Q = 1$). It also follows that $\mu_{\underline\omega} = V\alpha\wedge \alpha^{123}$. 
    
    The $1$-forms $\alpha,\alpha^1,\alpha^2,\alpha^3$ give a basis of the cotangent space of $X$ at any point. We let $\bm v, \bm v_1,\bm v_2,\bm v_3$ be its dual basis on the tangent space. We can directly compute the inner products between these vectors by the formula~\eqref{g-formula} for $g_{\underline{\omega}}$. This gives the claimed expression for $g_{\underline{\omega}}$.
\end{proof}

At this stage we haven't checked that the data $\alpha, \sigma$ is actually smooth. We will do this in the next Lemma. First, we introduce some notation. Write $Y = X/S^1$. Since the $S^1$-action is free, $Y$ is a compact 3-manifold and the quotient map $X \to Y$ is a principal $S^1$-bundle. We write $\mathcal{V} \leq TX$ for the vertical tangent bundle. We use the metric $g_{\underline{\omega}}$ to take the horizontal complement $H$ of $\mathcal{V}$. Since $H$ is $S^1$-invariant (because the metric itself is) it defines an $S^1$-connection in $X \to Y$. 

\begin{lemma}
    The $1$-forms $\alpha, \alpha^1,\alpha^2,\alpha^3$ and the $3\times 3$ matrix-valued function $\left(\sigma_{ij}\right)$ uniquely determined by $\underline\omega$ and $\bm v$ are smooth.  
\end{lemma}

\begin{proof}
    Since $\bm v$ is smooth, the 1-forms $\alpha^i=\iota_{\bm v}\omega_i$ are also smooth for $i=1,2,3$. Write $A_{\underline\omega} \colon TX \to \mathcal{V}$ for the vertical projection map of the connection in $X \to Y$. Then, for $\bm w\in TX$,
    \begin{align}\label{connection-1-form}
        A_{\underline\omega}\left(\bm w\right)=\alpha\left(\bm w\right)\bm v.
    \end{align}
    This can be seen as follows: for $\bm w\in H_{\underline\omega}$, $g_{\underline\omega}\left(\bm w,\bm v\right)=V^{-1}\alpha\left(\bm w\right)=0$ by Lemma \ref{lemma:metric-tensor} and the fact that $\alpha^i\left(\bm v\right)=0$; for $\bm w=\bm v$, $A_{\underline\omega}\left(\bm v\right)=\bm v=\alpha\left(\bm v\right)\bm v$. Now since ${\bm v}$ is smooth and nowhere vanishing, and $A$ is smooth, it follows that $\alpha$ is smooth as well.

    The formula $\mu_{\underline\omega} = V \alpha \wedge \alpha^{123}$ shows that $V$ is smooth and now the formula for $Q$ implies that $\sigma = VQ$ is also smooth. 
\end{proof}

\begin{definition}[The structural data]
The tuple  $\big\{ \alpha, \alpha^i, \left(\sigma_{ij}\right)\big\}$, uniquely associated to $\underline{\omega}$ in Lemma \ref{representation:canonical}, is called the \emph{structural data} of $\underline\omega$.
\end{definition}

\begin{lemma}\label{alphas-pulled-back}
We have $\cL_{\bm v} \alpha = 0 = \cL_{\bm v} \alpha_i$.
\end{lemma}
\begin{proof}
Recall that $\alpha^i = \iota_{\bm v} \omega_i$. So $\iota_{\bm v} \alpha^i=0$. Since $\cL_{\bm v} \omega_i = 0$ we see that $\diff \alpha_i = 0$ and hence $\cL_{\bm v} \alpha^i = 0$. Next note that, since $\alpha$ is a connection 1-form in $X \to Y$, $\diff \alpha$ is pulled back from $Y$. It follows that $\iota_v \diff \alpha =0$. At the same time, $\iota_{\bm v} \alpha = 1$ so $\diff (\iota_{\bm v} \alpha) =0$ and hence $\cL_{\bm v} \alpha =0$.
\end{proof}

\begin{remark}\label{SL3R-action}
    If $P\in \SL(3,\mathbf{R})$ is a constant matrix, then $\widetilde\omega_i=P_{ij}\omega_j$ is a new hypersymplectic structure on $X$ whose structural data is related to that of $\underline{\omega}$ by 
    \begin{align*}
    \widetilde\omega_i 
    = 
    \alpha\wedge \widetilde\alpha^i 
    + 
    \sum_{(pqr)=(123)} \widetilde\sigma_{ip}\widetilde\alpha^q\wedge \widetilde\alpha^r
    \end{align*}
    with 
    \begin{align*}
    \widetilde\alpha^i 
    = 
    P_{ij}\alpha^j, \;\; 
    \widetilde\sigma_{ip}
    = 
    \left( P\sigma P^T \right)_{ip}.
    \end{align*}
\end{remark}

\begin{example}\label{example}
In \cite{FHY}, a special type of hypersymplectic structure was introduced  on $\mathbb{T}^4$, which is invariant under a $\mathbb{T}^3$-action. Using affine coordinates $(x^0,x^1, x^2,x^3) \mapsto (e^{ix_0},e^{ix_1},e^{ix_2},e^{ix_3})$ on $\mathbb{T}^4$, such structures have the shape
\begin{align*}
\omega_i = \diff x^0 \wedge \alpha_{ip}\diff x^p + \frac{1}{2} \epsilon_{ijk}\diff x^{j}\wedge \diff x^k,
\end{align*}
where $\left( \alpha_{ij}\right)$ is a symmetric $3\times 3$ positive definite matrix-valued function of $x^0$. This structure is $\mathbb{T}^3$-invariant where $\mathbb{T}^3$ acts as 
\begin{align*}
    \left(e^{it_1}, e^{it_2}, e^{it_3}\right)\cdot \left( e^{ix_0}, e^{ix_1}, e^{ix_2}, e^{ix_3}\right) = \left( e^{ix_0}, e^{i(t_1+ x_1)}, e^{i(t_2+x_2)}, e^{i(t_3+x_3)}\right). 
\end{align*}
The structure $\underline\omega$ is $S^1$-invariant, for three different $S^1$-actions, acting in the obvious way on the three $S^1$ factors of $\mathbb{T}^4$. It fits into the current framework. For instance, if we take the action generated by $\bm v=\frac{\partial}{\partial x^3}$, then the structural data of $\underline{\omega}$ is
\begin{equation*}
\begin{split}
& \alpha 
 = 
\diff x^3 + \frac{\alpha_{13}}{\alpha_{33}} \diff x^1 + \frac{\alpha_{23}}{\alpha_{33}} \diff x^2,\\
&\alpha^1 
 = 
- \diff x^2 -\alpha_{13}\diff x^0,\\
&\alpha^2 
 = 
\diff x^1- \alpha_{23}\diff x^0,\\
&\alpha^3 
 = 
-\alpha_{33} \diff x^0,\\
&\sigma_{ij}
=
 \frac{\alpha_{ij}}{\alpha_{33}}. 
\end{split}
\end{equation*}
\end{example}

\section{Linear interpolation of hypersymplectic structures}
\subsection{Uniformizing hypersymplectic structures invariant under a free $S^1$-action}

We begin by showing that, when the circle action is free, the 4-manifold $X$ is diffeomorphic to $\mathbb{T}^4$.

\begin{theorem}
\label{thm:standardization}
Let $X^4$ be a compact smooth $4$-manifold admitting a smooth hypersymplectic structure which is invariant under a free smooth $S^1$-action, then $Y^3:=X^4/S^1$ is diffeomorphic to $\mathbb{T}^3$ and $X^4$ is $S^1$-equivariantly diffeomorphic to $\mathbb{T}^4$, where $S^1$ acts on $\mathbb{T}^4$ by rotation on the first circle factor. 
\end{theorem}

\begin{proof}
By Lemma~\ref{alphas-pulled-back}, the forms $\alpha_i$ are $S^1$-invariant and so are pulled back to $X$ from $Y$. This means that  the symmetric $(0,2)$ tensor
    \begin{align*} 
    g:=\alpha^1\otimes \alpha^1+\alpha^2\otimes \alpha^2+\alpha^3\otimes \alpha^3
    \end{align*}
    defines a Riemannian metric on the compact manifold $Y$. The fact $\diff\alpha^i = 0$ implies $g$ is a flat Riemannian metric and $\left\{\alpha_1,\alpha_2,\alpha_3\right\}$ is a global parallel coframe. This means $\left(Y^3, g\right)$ has trivial holonomy group. By the classical theorems of Bieberbach and Auslander-Kuranish \cite{AK} (which imply $\pi_1(Y)$ is generated by three linearly independent pure translations of $\mathbb{R}^3$), or the classification of compact flat Riemannian $3$-manifolds \cite{Sc}, we conclude that $Y$ is diffeomorphic to $\mathbb{T}^3$. 

We now rule out the possibility that $X \to \mathbb{T}^3$ is a non-trivial $S^1$-bundle. We first show that $b_1(X)=4$. For this we use a deep result from 4-dimensional symplectic topology, proved in \cite{Li} which states that for a compact symplectic 4-manifold with $c_1=0$,
	\begin{itemize}
	\item $b^+ \leq 3$;
	\item $b_1 \leq 4$; and
	\item the signature $\tau$ is $-16$, $-8$ or $0$. 
	\end{itemize}
The $\omega_i$'s are 3 linearly independent closed self-dual 2-forms and so we see that $b^+ =3$. Since ${\bm v}$ is nowhere vanishing, $0= \chi = 2 - 2b_1 + b_2$ and so from $b_1 \leq 4$ we conclude $b_2 \leq 6$. This means that $\tau =0$ and so $b^-=3$, hence $b_2=6$ and so $b_1=4$. 

We now complete the proof by showing that the first Chern class of the principal $S^1$-bundle $X \to Y$ vanishes. Applying the Gysin sequence of cohomology groups (with integral coefficients) to the principal bundle $\pi \colon X \to Y$ we obtain the following long exact sequence
    \begin{align}
    \label{Gysin-sequence}
      H^1(Y)\xrightarrow{\pi^*} H^1(X) \xrightarrow{\pi_*} H^0(\mathbb{T}^3)\xrightarrow{c_1\cup~\cdot} H^2(\mathbb{T}^3) \xrightarrow{\pi^*} H^2(X)\rightarrow \cdots
    \end{align}
   where $c_1=c_1(X \to Y)\in H^2(Y)$. If $c_1\neq 0$, then the map $c_1\cup\cdot: H^0(\mathbb{T}^3)\xrightarrow{} H^2(\mathbb{T}^3)$ is injective which implies the previous map $\pi_*: H^1(X)\xrightarrow{} H^0(\mathbb{T}^3)$ is zero. This leads to a contradiction since $\pi^*: H^1(\mathbb{T}^3)\to H^1\left(X \right)$ is surjective and $b_1(\mathbb{T}^3)=3$ but $b_1(X)=4$. 
  \end{proof}

  \begin{remark}
      It is shown by Friedl and Vidussi \cite[Theorem 3]{FV} that a compact $4$-manifold admitting a free $S^1$-action supports a symplectic structure with trivial canonical class if and only if it a $\mathbb{T}^2$-bundle over $\mathbb{T}^2$. Bowden \cite{B} completes the characterisation of which symplectic $4$-manifolds admit non-trivial circle actions. It should also be noted that Hajduk and Walczak \cite[Corollary 2.1]{HW} showed that any symplectic structure on $\mathbb{T}^4$ invariant under a free linear $S^1$-action is isotopic to a standard constant coefficient symplectic structure. The above theorem can be seen as a refinement of these results in the special case of circle invariant (with effective action) hypersymplectic structure. 
  \end{remark}
    
\subsection{Linear interpolation to a hyperk\"ahler structure}

We now explain how to interpolate between an $S^1$-invariant hypersymplectic structure on $\mathbb{T}^4$ and a hyperk\"ahler structure. We will show later how to ensure this can be done without changing the cohomology classes of the $\omega_i$'s. 

Let $S^1$ act on $\mathbb{T}^4$ in the standard way that rotates the first factor, giving a trivial principal $S^1$-bundle $\pi \colon \mathbb{T}^4 \rightarrow \mathbb{T}^3$. We use the coordinate $\theta$ on the first factor of $\mathbb{T}^4$. Let $\underline\omega$ be a hypersymplectic structure invariant under this $S^1$-action, with canonical decomposition 
    \begin{align}
    \label{canonical-decomposition-on-torus}
        \omega_i
        =
        \alpha\wedge \alpha^i + \sum_{(pqr)=(123)} \sigma_{ip}\alpha^q\wedge \alpha^r, \; \; i=1,2,3,
    \end{align}  
and $\alpha,\alpha^1,\alpha^2,\alpha^3$ being positively oriented.  We have $\alpha=\diff \theta+\beta$ for some $1$-form $\beta$ pulled back from $\mathbb{T}^3$.     

Consider any constant $(3\times 3)$-matrix $\mathbf{B}$ with positive-definite symmetric part $\widehat{\mathbf{B}}$, and with its skew-symmetric part denoted by $\check{\mathbf{B}}$. The $S^1$-invariant triple $\underline\omega^{\mathbf{B}}=\left(\omega_1^\mathbf{B}, \omega_2^\mathbf{B}, \omega_3^\mathbf{B}\right)$ with 
\begin{align}
\label{B-hyperkahler}
\omega_i^\mathbf{B}
& := 
\diff \theta \wedge \alpha^i 
+
\sum_{(pqr)=(123)}\mathbf{B}_{ip}\alpha^q\wedge\alpha^r\nonumber\\
& =
\left(\diff \theta+ \check{\mathbf{B}}_{23}\alpha^1 + \check{\mathbf{B}}_{31}\alpha^2 + \check{\mathbf{B}}_{12}\alpha^3\right)\wedge \alpha^i 
+ 
\sum_{(pqr)=(123)}\widehat{\mathbf{B}}_{ip}\alpha^q\wedge\alpha^r,
\end{align}
satisfies
\begin{align*}
    \omega_i^\mathbf{B}\wedge\omega_j^\mathbf{B}
    = 
    2\widehat{\mathbf{B}}_{ij}\alpha^\mathbf{B}\wedge\alpha^1\wedge\alpha^2\wedge\alpha^3
\end{align*}
where $\alpha^\mathbf{B}=\diff \theta+ \check{\mathbf{B}}_{23}\alpha^1 + \check{\mathbf{B}}_{31}\alpha^2 + \check{\mathbf{B}}_{12}\alpha^3$, and therefore $\underline\omega^\mathbf{B}$ is a hyperk\"ahler structure on $\mathbb{T}^4$. This is because the $Q$-matrix for $\underline\omega^\mathbf{B}$ is $\widehat{\mathbf{B}}/\left(\det \mathbf{B}\right)^\frac{1}{3}$ is a \emph{constant} matrix on $\mathbb{T}^4$. So after acting on the $\omega_i^{\mathbf{B}}$ by a constant matrix $P \in \SL(3,\mathbf{R})$ we obtain a triple $\tilde{\omega}_i^{\mathbf{B}} = P_{ij} \omega_j^{\mathbf{B}}$ for which $\tilde{\omega}_i^{\mathbf{B}} \wedge \tilde{\omega}_j^{\mathbf{B}} = 2\delta_{ij} \mu$ for some volume form $\mu$. It is a standard fact that this implies that $\left(\tilde{\omega}_{1}^{\mathbf{B}}, \tilde{\omega}_{2}^{\mathbf{B}}, \tilde{\omega}_{3}^{\mathbf{B}}\right)$ is a hyperk\"ahler triple. (See, e.g., \cite[Propositions 3.7 and 3.9]{FY}.)

\begin{proposition}\label{linear-interp}
The linear interpolation
\begin{equation}
    \underline\omega^\mathbf{B}(s)
    := 
    \left(1-s\right)\underline\omega
    + 
    s\underline\omega^\mathbf{B}, \; s\in [0,1],
\end{equation}
is a path of hypersymplectic structures joining $\underline{\omega}$ to the hyperk\"ahler structure $\underline{\omega}^{\mathbf{B}}$.
\end{proposition}
\begin{proof}
The path $\underline\omega^\mathbf{B}(s)$ is explicitly given by
\begin{align}
\omega_i^\mathbf{B}(s)
& = 
\left( \left(1-s\right)\alpha + s\left(\diff\theta+ \check{\mathbf{B}}_{23}\alpha^1 + \check{\mathbf{B}}_{31}\alpha^2 + \check{\mathbf{B}}_{12}\alpha^3\right)\right)\wedge \alpha^i \nonumber \\
&\phantom{=}\qquad
+ 
\sum_{(pqr)=(123)}\left( \left(1-s\right)\sigma_{ip}+ s\widehat{\mathbf{B}}_{ip}\right)\alpha^q\wedge\alpha^r,
\end{align}
which is a path of triples of $S^1$-invariant closed $2$-forms on $\mathbb{T}^4$. Notice that
\begin{align*}
\omega_i^\mathbf{B}(s)\wedge\omega_j^\mathbf{B}(s)
= 
2\sigma_{ij}^\mathbf{B}(s)\alpha^\mathbf{B}(s)\wedge\alpha^1\wedge\alpha^2\wedge\alpha^3,
\end{align*}
where $\sigma^\mathbf{B}(s):=\left(1-s\right)\sigma+s\widehat{\mathbf{B}}>0$ as $\sigma$ and $\widehat{\mathbf{B}}$ are both symmetric and positive-definite everywhere and $\alpha^\mathbf{B}(s):=\left(1-s\right)\alpha + s\left(\diff\theta+ \check{\mathbf{B}}_{23}\alpha^1 + \check{\mathbf{B}}_{31}\alpha^2 + \check{\mathbf{B}}_{12}\alpha^3\right)$ satisfy
\[ 
\alpha^\mathbf{B}(s)\wedge\alpha^1\wedge\alpha^2\wedge\alpha^3
=
\diff \theta\wedge\alpha^1\wedge\alpha^2\wedge\alpha^3>0
\]
with respect to the orientation determined by $\underline\omega$.
\end{proof}

\subsection{Preserving the cohomology classes}\label{choice-of-B}
In order for $\underline\omega^\mathbf{B}(s)$ to stay inside the cohomology class of $\underline\omega$, we need to choose $\mathbf{B}$ suitably. 

\begin{definition}[Period of hypersymplectic structure]
    The \emph{period} of a hypersymplectic structure $\underline\omega=\left(\omega_1, \omega_2, \omega_3\right)$ on $X^4$ is defined to be $\left( [\omega_1], [\omega_2], [\omega_3]\right)\in H^2\left(X^4;\mathbf{R}\right)\otimes \mathbf{R}^3$.
\end{definition}

In the particular case $X^4=\mathbb{T}^4$, let 	
\[
g_0
=\mathrm{d}\theta\otimes \mathrm{d}\theta 
+ 
\alpha^1\otimes \alpha^1 + \alpha^2\otimes \alpha^2+\alpha^3\otimes \alpha^3
\]
be a flat Riemannian metric on $X$, then 
\[
\mathrm{d}\theta \wedge \alpha^1, \;
\mathrm{d}\theta \wedge \alpha^2, \;
\mathrm{d}\theta \wedge \alpha^3,\;
\alpha^2\wedge\alpha^3,\;
\alpha^3\wedge \alpha^1,\;
\alpha^1\wedge \alpha^2
\]
forms a basis for the space of real harmonic $2$-forms which can be seen by considering the matrix of their cup-products. We now compute the following:
\begin{align*}
\left(\left[\omega_i\right]\cup \left[\mathrm{d}\theta\wedge \alpha^j\right], \left[\mathbb{T}^4\right]\right)
& = 
\int_{\mathbb{T}^4} \sigma_{ij}\mathrm{d}\theta \wedge \alpha^{123}
+ 
\int_{\mathbb{T}^4} \mathrm{d}\theta \wedge\alpha\wedge\alpha^i\wedge\alpha^j,\; i,j=1,2,3, \\
\left(\left[\omega_i^\mathbf{B}\right]\cup \left[\mathrm{d}\theta\wedge\alpha^j\right], \left[\mathbb{T}^4\right]\right)
& = 
\mathbf{B}_{ij}\int_{\mathbb{T}^4} \mathrm{d}\theta \wedge\alpha^{123},\; i,j=1,2,3,\\
\left( \left[\omega_i\right]\cup \left[\alpha^q\wedge\alpha^r\right], \left[\mathbb{T}^4\right]\right)
& = 
\delta_{ip}\int_{\mathbb{T}^4} \mathrm{d}\theta \wedge \alpha^{123}, \; i=1,2,3,\; (pqr)=(123),\\
\left( \left[\omega_i^\mathbf{B}\right]\cup \left[\alpha^q\wedge\alpha^r\right], \left[\mathbb{T}^4\right]\right)
& = 
\delta_{ip}\int_{\mathbb{T}^4} \mathrm{d}\theta \wedge\alpha^{123},\; i=1,2,3, \; (pqr)=(123).
\end{align*}
These computations (together with the non-degeneracy of the cup product on $H^2$) show that we can ensure that
\[
\left(\left[\omega_1\right],\left[\omega_2\right],\left[\omega_3\right]\right) 
= 
\left( \left[\omega_1^\mathbf{B}\right],\left[\omega_2^\mathbf{B}\right],\left[\omega_3^\mathbf{B}\right]\right)
\]
if we simply set 
\begin{equation}
\label{def:B}
\mathbf{B}_{ij}
= 
\left(\int_{\mathbb{T}^4} \sigma_{ij}\mathrm{d}\theta \wedge \alpha^{123}
+ 
\int_{\mathbb{T}^4} \mathrm{d}\theta \wedge\alpha\wedge\alpha^i\wedge\alpha^j\right)/\int_{\mathbb{T}^4}\mathrm{d}\theta\wedge\alpha^{123}.
\end{equation}
Moreover, the matrix $\widehat{\mathbf{B}}$ is positive definite since
\[
2\widehat{\mathbf{B}}_{ij}
= \mathbf{B}_{ij}
+ 
\mathbf{B}_{ji}
= 
2\int_{\mathbb{T}^4}\sigma_{ij}\mathrm{d}\theta \wedge\alpha^{123}/\int_{\mathbb{T}^4}\mathrm{d}\theta\wedge\alpha^{123}
\]
and $\left(\sigma_{ij}\right)>0$ on $\mathbb{T}^4$. 

At this point, we have proved Theorem~\ref{main-theorem} in the case when the $S^1$-action is free.

\begin{remark}
      One hypersymplectic structure may be invariant under more than one free $S^1$-actions (as shown in the Example \ref{example} of hypersymplectic structures in symmetric normal form). The construction above will lead to isotopies to possibly different hyperk\"ahler structures (in the fixed cohomology class) if we are using different actions.
  \end{remark}

\subsection{An effective action must be free}
Let $X$ be a smooth compact $4$-manifold with an effective circle action $\left\{\Phi_t\right\}_{t\in \mathbb{T}^1}$,  that is, there exists an injective homomorphism $\Phi:\mathbb{T}^1\rightarrow\Diff\left(X\right)$ sending $t$ to $\Phi_t$, and $\underline\omega$ is a smooth hypersymplectic structure invariant under this circle action. Let $\bm v$ be the vector field generating such action.

\begin{lemma}
\label{lemma:isolated-zero}
The zeros of $\bm v$ are isolated, and the index of each zero is $1$.
\end{lemma}

\begin{proof}
    Let $p$ be a zero of $\bm v$, then $\diff \Phi_t|_p: T_p X\to T_p X$ is an isometry for any $t\in \mathbb{T}^1$ when $T_p X$ is equipped with the metric $g_{\underline\omega}|_p$. Moreover, $T_p X$ is naturally oriented by $\underline\omega$ and $\diff\Phi_t|_p$ must preserve the orientation. Since $\diff\Phi_t|_p$ preserves $\underline\omega|_p$, it must also preserve the triple \[\theta_i|_p = Q_{ij}^{-\frac{1}{2}}\omega_j|_p\] 
which is hyperk\"ahler at $p$. This further implies that $\left\{ \diff\Phi_t|_p\right\}_{t\in \mathbb{T}^1}$ is a subgroup of $\SU(2)$. It follows that there is an integer $n \in \N$ and an isomorphism $T_pX \cong \C^2$ of $\SU(2)$-representations, in which
\[
\diff \Phi_t|_p = \begin{pmatrix} t^n & 0 \\ 0 & t^{-n} \end{pmatrix}.
\]

To prove the lemma we must show that $n=1$. If $n \neq 1$ then there is some $t \neq \id$ with $\diff \Phi_t|_p = \id$. So $\Phi_t$ fixes both $p$ and the tangent space at $p$ and so it fixes all geodesics emanating from $p$. This forces $\Phi_t = \id$ (by a continuity argument) contradicting the effectiveness of the action.
%
%
    \end{proof}

\begin{proposition}
Let $X$ be a smooth compact 4-manifold, with a hypersymplectic structure $\underline{\omega}$ that is invariant under an effective $S^1$-action. Then in fact the action is free.
\end{proposition}

\begin{proof}

Theorem 7.7 of \cite{Li} shows that the compact symplectic 4-manifold with $c_1=0$ has the same rational homology as one of the following: a $\mathbb{T}^2$-bundle over $\mathbb{T}^2$, an Enriques surfaces or a $K3$ surface. The existence of a hypersymplectic structure implies $b_+=3$ and thus the Enriques surfaces are ruled out. Meanwhile, a theorem of Atiyah--Hirzebruch \cite{AH} asserts that a spin 4-manifold with an effective $S^1$-action has vanishing signature. This rules out the case of a $K3$ surface. So $X$ must have the same rational homology as a $\mathbb{T}^2$-bundle over $\mathbb{T}^2$ and in particular it must have zero Euler characteristic. In view of Lemma~\ref{lemma:isolated-zero}, the Poincar\'e-Hopf Index Theorem implies that $\bm v$ has no zeros and so the action is free as claimed.
\end{proof}

This Proposition completes the proof of Theorem~\ref{main-theorem}.
%
%
%

\end{document}